\documentclass{llncs}

\usepackage{amsmath,amsfonts,mathdots,proof,combelow}

\newcommand{\LL}{\mathbf{L}}
\newcommand{\FL}{\mathbf{FL}}
\newcommand{\Lomega}{\mathbf{L}^{\!+}_{\omega}}
\newcommand{\Linfty}{\mathbf{L}^{\!+}_{\infty}}
\newcommand{\ACTomega}{\mathbf{ACT}_{\omega}}
\newcommand{\Lcirc}{\mathbf{L}^{\!+}_{\mathrm{circ}}}

\newcommand{\PR}{\mathrm{Pr}}
\newcommand{\Tp}{\mathrm{Tp}}

\newcommand{\TOTAL}{\mbox{\sc total}}
\newcommand{\ALT}{\mbox{\sc alt}}

\newcommand{\BS}{\mathop{\backslash}}
\newcommand{\SL}{\mathop{/}}

\newcommand{\Gc}{\mathcal{G}}
\newcommand{\Lc}{\mathcal{L}}
\newcommand{\Uc}{\mathcal{U}}
\newcommand{\Nc}{\mathcal{N}}

\newcommand{\Top}{\mathrm{top}}
\newcommand{\IS}{\mathrm{is}}
\newcommand{\FG}{\mathbb{F}}
\newcommand{\FGI}[1]{[\![#1]\!]}

\begin{document}

\title{The Lambek Calculus with Iteration: Two Variants}

\titlerunning{The Lambek Calculus with Iteration}  
%
\author{Stepan Kuznetsov\thanks{This work is supported by the Russian Science Foundation under grant \mbox{16-11-10252.}
Accepted for presentation at WoLLIC 2017 and publication in Springer LNCS.}}
\institute{Steklov Mathematical Institute of RAS, Moscow, Russia\\
\email{sk@mi.ras.ru}}

\maketitle              

\begin{abstract}
Formulae of the Lambek calculus are constructed using three
binary connectives, multiplication and two divisions. We extend
it using a unary connective, positive Kleene iteration.
For this new operation, following its natural interpretation,
we present two lines of calculi. The first one is a fragment 
of infinitary action logic and includes an omega-rule for
introducing iteration to the antecedent. We also consider
a version with infinite (but finitely branching) derivations
and prove equivalence of these two versions. In Kleene algebras, 
this line of calculi corresponds to the *-continuous case. 
For the second line, we restrict our infinite derivations to 
cyclic (regular) ones. 
We show that this system is equivalent to a variant of action
logic that corresponds to general residuated Kleene algebras, not 
necessarily *-continuous. Finally, we show that, in contrast 
with the case without division operations (considered by Kozen), 
the first system is strictly stronger than the second one. 
To prove this, we use a complexity argument. Namely, we show, 
using methods of Buszkowski and Palka, that the first system is 
$\Pi_1^0$-hard, and therefore is not recursively enumerable and 
cannot be described by a calculus with finite derivations.
\end{abstract}

\section{The Infinitary Lambek Calculus with Positive Iteration}

The Lambek calculus $\LL$~\cite{Lambek58} deals with formulae that are built using three
connectives, $\cdot$ (product), $\BS$, and $\SL$ (left and right divisions).  These connectives
enjoy a natural interpretation as operations on formal languages (completeness shown
by Pentus~\cite{Pentus95}). There are, however, also other interesting and well-respected
operations on formal languages, and it is quite natural to try to extend $\LL$ by adding
these operations as new connectives.

One of the most common of such operations is {\em iteration,} or 
{\em Kleene star:} for a language $M$ over an alphabet $\Sigma$ its iteration is defined
as follows: $$M^* = \{ u_1 \dots u_n \mid n \geq 0, u_i \in M \}.$$
As one can notice, $M^*$ always includes the empty word, $\varepsilon$. The original
Lambek calculus, however, obeys so-called {\em Lambek's non-emptiness restriction,} that
is, the empty sequence is never allowed in $\LL$ (this restriction is motivated by linguistic
applications; from the algebraic point of view, this means that we're considering residuated
semigroups instead of residuated monoids). Therefore, throughout this paper we consider
a modified version of Kleene star, called {\em positive iteration:}
$$M^+ = \{ u_1 \dots u_n \mid n \geq 1, u_i \in M \} = M^* - \{ \varepsilon \}.$$

In this paper, we introduce several extensions of the Lambek calculus with this new connective,
establish connections between them, and prove some complexity bounds.

Formulae of the Lambek calculus with positive iteration, usually called {\em types,} are built
from a countable set of variables (primitive types) $\PR = \{ p_1, p_2, p_3, \dots \}$ using
three binary connectives, $\cdot$, $\BS$, and $\SL$, and one unary connective, ${}^+$
(written in the postfix form, $A^+$). The set of all types is denoted by $\Tp$.
Types are denoted by capital Latin letters; capital Greek letters
stand for finite linearly ordered sequences of types.

Derivable objects are {\em sequents} of the form $\Pi \to A$, where $A \in \Tp$ and
$\Pi$ is a {\em non-empty} finite sequence of types.

Now let's define the first calculus for positive iteration, $\Lomega$. The axioms and the rules
for $\cdot$, $\BS$, and $\SL$ are the same as in the original Lambek calculus $\LL$:

$$
\infer[(\mathrm{ax})]{A \to A}{}
$$
$$
\infer[(\to\BS)\mbox{, where $\Pi$ is non-empty}]{\Pi \to A \BS B}{A, \Pi \to B}
\qquad
\infer[(\BS\to)]{\Gamma, \Pi, A \BS B, \Delta \to C}{\Pi \to A & \Gamma, B, \Delta \to C}
$$
$$
\infer[(\to\SL)\mbox{, where $\Pi$ is non-empty}]{\Pi \to B \SL A}{\Pi, A \to B}
\qquad
\infer[(\SL\to)]{\Gamma, B \SL A, \Pi, \Delta \to C}{\Pi \to A & \Gamma, B, \Delta \to C}
$$
$$
\infer[(\to\cdot)]{\Gamma, \Delta \to A \cdot B}{\Gamma \to A & \Delta \to B}
\qquad
\infer[(\cdot\to)]{\Gamma, A \cdot B, \Delta \to C}{\Gamma, A, B, \Delta \to C}
$$

For ${}^+$, this calculus includes a countable set of right rules:
$$
\infer[(\to{}^+)_n\mbox{, for $n \ge 1$}]{\Pi_1, \dots, \Pi_n \to A^+}{\Pi_1 \to A & \ldots & \Pi_n \to A}
$$
and one left rule
$$
\infer[({}^+\to)_\omega]
{\Gamma, A^+, \Delta \to C}
{\Gamma, A, \Delta \to C & \Gamma, A, A, \Delta \to C & \Gamma, A, A, A, \Delta \to C & \ldots}
$$
This rule is an {\em $\omega$-rule,} or an {\em infinitary} rule. Application of such a rule makes the proof tree
infinite. This is somewhat unpleasant from the computational point of view, but, as we show later on, it appears
to be inevitable.

The rules $(\to{}^+)_n$ and $({}^+\to)_\omega$ come from the rules for iteration in {\em infinitary action logic}, 
$\ACTomega$~\cite{Buszkowski07}. Our system $\Lomega$ differs from
$\ACTomega$ in the following two points.
\begin{enumerate}
\item $\Lomega$ enriches the ``pure'' (multiplicative) Lambek calculus $\LL$, while $\ACTomega$ is based on the full Lambek calculus $\FL$,
including also additive conjunction ($\wedge$) and disjunction ($\vee$). This means that complexity lower bounds for $\Lomega$ are stronger
results than lower bounds for $\ACTomega$.
\item In contrast to $\ACTomega$, in $\Lomega$ we have Lambek's non-emptiness restriction, and therefore use positive iteration instead of Kleene star.
\end{enumerate}

The cut rule of the form
$$
\infer[(\mathrm{cut})]{\Gamma, \Pi ,\Delta \to C}{\Pi \to A & \Gamma, A, \Delta \to C}
$$
is admissible in $\Lomega$. This fact is proved by the same transfinitary cut-elimination procedure, as presented by Palka~\cite{Palka07}
 for $\ACTomega$ (for a restricted fragment of $\Lomega$ cut elimination was independently shown by Ryzhkova~\cite{Ryzhkova13}).

The admissibility of $(\mathrm{cut})$ yields the fact that the rules $(\to\BS)$, $(\to\SL)$, $(\cdot\to)$, and, most interestingly, $({}^+\to)_\omega$ are
invertible. 

The Lambek calculus $\LL$, defined by axioms $(\mathrm{ax})$ and rules $(\to\BS)$, $(\BS\to)$, $(\to\SL)$, $(\SL\to)$, $(\to\cdot)$, and $(\cdot\to)$, is
a conservative fragment of $\Lomega$. Cut elimination for $\LL$ was known already by Lambek~\cite{Lambek58}.

The calculus $\Lomega$ defined in this section is sound with respect to the intended interpretation on formal languages, where ${}^+$ is interpreted
as positive iteration:
$$
M^+ = \{ u_1 \dots u_n \mid n \geq 1, u_i \in M \},
$$
and the Lambek connectives are interpreted in the same way as for $\LL$:
\begin{align*}
M \cdot N &= \{ uv \mid u \in M, v \in N \},\\
M \BS N &= \{ u \in \Sigma^+ \mid (\forall v \in M) \, vu \in N \}, \\
N \SL M &= \{ u \in \Sigma^+ \mid (\forall v \in M) \, uv \in N \}.
\end{align*}
The arrow, $\to$, is interpreted as the subset relation.

Completeness with respect to this interpretation is an open problem. 

\section{$\Pi_1^0$-completeness of $\Lomega$}

In this section we prove that derivability in $\Lomega$ is $\Pi_1^0$- (co-r.e.-) hard. Basically, we follow the same strategy as Buszkowski~\cite{Buszkowski07},
namely, encoding the totality problem for context-free grammars. Our construction, however, is more involved: instead of embedding context-free grammars into
the Lambek environment as Ajdukiewicz -- Bar-Hillel basic categorial grammars, we use another translation by Safiullin~\cite{Safiullin07} which yields a categorial
grammar that assigns {\em exactly one type} to each letter of the alphabet. This trick allows us to avoid using additive operations, and prove the complexity
lower bound for the extension of the original, purely multiplicative Lambek calculus $\LL$.
For the purely multiplicative fragment of $\ACTomega$, the lower complexity bound was left as an open problem by Buszkowski in~\cite{Buszkowski07}. Here we solve not
that problem exactly, but its version with Lambek's restriction.

Throughout this paper, all languages do not contain the empty word. Accordingly, all context-free grammars do not contain $\varepsilon$-rules.
By $\TOTAL^+$ we denote the set of all context-free grammars $\Gc$ such that the language generated by $\Gc$ is the set of all non-empty words, $\Sigma^+$.
The problem $\TOTAL^+$ is $\Pi_1^0$-hard.  Indeed, as shown in~\cite{Buszkowski07}, $\TOTAL$, the totality problem for context-free grammars possibly
using the empty word, reduces to $\TOTAL^+$. In its turn, $\TOTAL$ itself is known to be undecidable, and standard proofs of this fact actually yield more:
they reduce a well-known $\Sigma_1^0$- (r.e.-) complete problems, e.g., Post's correspondence problem~\cite[Theorem~9.22]{Hopcroft01} or
halting problem for Turing machines~\cite[Example~5.43]{Du01}, to the {\em complement} of $\TOTAL$. This makes $\TOTAL$ and $\TOTAL^+$ themselves $\Pi_1^0$-complete.

We can further restrict ourselves to context-free grammars over a two-letter alphabet, $\{ b, c \}$. Denote the $\varepsilon$-free totality problem 
over $\{ b, c\}$ by $\TOTAL^+_2$. The original problem $\TOTAL^+$ is reduced to $\TOTAL^+_2$ in the following way. Let $\Gc$ be a context-free
grammar that defines a language $\Lc(\Gc)$ over $\Sigma = \{ a_0, a_1, \dots, a_n \}$. The homomorphism $h \colon a_i \mapsto b^{i} c$ is a one-to-one
correspondence between $\Sigma^+$ and $\{ u \in \{ b,c \}^+ \mid \mbox{$u$ ends on $c$ and doesn't contain $b^{n+1}$ as a subword} \}$. Now we
can computably transform $\Gc$ into a new context-free grammar $\Gc'$ for the language $h(\Lc(\Gc)) \cup \{ u \in \{b,c\}^+ \mid \mbox {$u$ ends on $b$ or contains
$b^{n+1}$ as a subword} \}$ over $\{b,c\}$. Clearly, $\Gc \in \TOTAL^+ \iff \Gc' \in \TOTAL_2^+$. This establishes the necessary reduction and $\Pi_1^0$-hardness
of $\TOTAL_2^+$.

Finally, we consider the {\em alternation problem} for context-free grammars over $\{ b, c \}$, denoted by $\ALT_2$. A context-free grammar
$\Gc$ belongs to $\ALT_2$ if the language it generates includes the language $(\{b\}^+ \{c\}^+)^+ = \{ b^{m_1} c^{k_1} \ldots b^{m_n} c^{k_n} \mid
n \geq 1, m_i \geq 1, k_j \geq 1 \}$ (as a subset). Clearly, $\ALT_2$ is also $\Pi_1^0$-hard by reduction of $\TOTAL_2^+$, since
$M = \{ b, c \}^+ \iff \{b\} \cdot M \cdot \{c\}   \supseteq (\{b\}^+ \{c\}^+)^+$.

Now we need an encoding of context-free grammars in the Lambek calculus. A {\em Lambek categorial grammar with unique type assignment} over
the alphabet $\{ a_1, \dots, a_n \}$ consists of $(n+1)$ types (without the $^+$ connective)
 $A_1, \dots, A_n, H$ ($H$ is called the {\em target} type), and a word $w = a_{i_1} \dots a_{i_m}$ belongs to the
language generated by this grammar if{f} the sequent $A_{i_1}, \dots, A_{i_m} \to H$ is derivable in $\LL$. These grammars have the same expressive
power as context-free grammars (without $\varepsilon$-rules).

\begin{theorem}[A. Safiullin, 2007]\label{Th:Safiullin}
For every context-free language there exists, and can be effectively constructed from the original context-free grammar,
a Lambek categorial grammar with unique type assignment.~\rm{\cite{Safiullin07}}
\end{theorem}
The inverse translation, from Lambek categorial grammars to context-free grammars, is also available due to Pentus~\cite{Pentus93} (in this paper we
don't need it). In order to make this paper logically self-contained, we revisit Theorem~\ref{Th:Safiullin} and give its full proof in the Appendix
(in Safiullin's paper~\cite{Safiullin07}, the proof is only briefly sketched).

Now we're ready to prove the main result of this section.

\begin{theorem}\label{Lomega_Pi}
The derivability problem for $\Lomega$ is $\Pi_1^0$-complete.
\end{theorem}

\begin{proof}
The fact that this problem belongs to class $\Pi_1^0$ (the upper bound) is established by the same argument as for $\ACTomega$ in~\cite{Palka07}.

To prove $\Pi_0^1$-hardness of the derivability problem in $\Lomega$ (the lower bound), we encode $\ALT_2$. For every context-free grammar $\Gc$ over
$\{b,c\}$ we algorithmically construct an $\Lomega$-sequent $E \to H$ such that
$$\Gc \in \ALT_2 \iff E \to H \text{ is derivable in $\Lomega$}.$$
First we apply Theorem~\ref{Th:Safiullin} to $\Gc$ and obtain a Lambek categorial grammar with unique type assignment. In this case,
it consists of three types, $B$, $C$, and $H$. Next, let $E = (B^+ \cdot C^+)^+$.

Now, since the $(\cdot\to)$ and $({}^+\to)$ rules are invertible, the sequent $E \to H$ is derivable in $\Lomega$ if{f} for any positive
natural numbers $n$, $m_1$, \dots, $m_n$, $k_1$, \dots, $k_n$ the sequent $B^{m_1}, C^{k_1}, \dots, B^{m_n}, C^{k_n} \to H$ is derivable
in $\Lomega$, and, since it doesn't contain $^{+}$, by conservativity also in the Lambek calculus $\LL$. By definition of Lambek grammar,
this is equivalent to $b^{m_1} c^{k_1} \dots b^{m_n} c^{k_n} \in \Lc(\Gc)$. Therefore, $E \to H$ is derivable if{f} the language generated by $\Gc$ 
includes all words of the form $b^{m_1} c^{k_1} \dots b^{m_n} c^{k_n}$, i.e., $\Gc \in \ALT_2$.
\end{proof}

\section{The Calculus with Infinite Derivation Branches}

In this section we define $\Linfty$, another infinitary calculus that extends $\LL$ with positive iteration, in the spirit of
sequent systems with non-well-founded derivations for other logics~\cite{Brotherston11}\cite{Mints78}\cite{Shamkanov14}. Compared to $\Lomega$,
$\Linfty$ has a finite number of rules and each rule has a finite number of premises. The tradeoff is that now derivation trees are allowed
to have infinite depth.

The Lambek part (rules for $\BS$, $\SL$, and $\cdot$) is taken from $\LL$. The rules for
positive iteration are as follows:
$$
\infer[(\to{}^+)_1]{\Pi \to A^+}{\Pi \to A}
\quad
\infer[(\to{}^+)_\mathrm{L}]{\Pi_1, \Pi_2 \to A^+}{\Pi_1 \to A & \Pi_2 \to A^+}
$$
$$
\infer[({}^+\to)_\mathrm{L}]{\Gamma, A^+, \Delta \to C}
{\Gamma, A, \Delta \to C & \Gamma, A, A^+, \Delta \to C}
$$

As said before, we allow infinitely deep derivations. For the cut-free version, any trees with possibly infinite paths are allowed, but for the
calculus with $(\mathrm{cut})$ one has to be extremely cautious. Clearly, allowing arbitrary infinite proofs would yield dead circles  without actually
using rules for $^+$:
$$
\infer[(\mathrm{cut})]{\vphantom{t} p \to q}{p \to p & 
\infer[(\mathrm{cut})]{\vphantom{t} p \to q}{p \to p & 
\infer[(\mathrm{cut})]{\vphantom{t} p \to q}{\iddots}}}
$$
Such ``derivations'' should be ruled out. There are, however, trickier cases like the following:
$$
\infer[({}^+\to)_\mathrm{L}]
{p^+ \to p}
{p \to p & \infer[(\mathrm{cut})]{p, p^+ \to p}
{\infer[(\to{}^+)_\mathrm{L}]{p, p^+ \to p^+}{p \to p & p^+ \to p^+} & 
\infer[({}^+\to)_\mathrm{L}]{p^+ \to p}{\iddots}}}
$$
Here in the only infinite path we can see an infinite number of $({}^+\to)$ applications. However,
the resulting sequent, $p^+ \to p$, is not valid under the formal language interpretation (e.g., $\{a\}^+ \not\subseteq \{a\}$) and
therefore should not be derivable.

For the calculus with $(\mathrm{cut})$, we impose the following constraint on the infinite derivation tree:
{\em in each infinite path there should be an infinite number of applications of $({}^+\to)_{\mathrm{L}}$ 
with \textbf{the same} active occurrence of $A^+$} (the occurrence is tracked by
individuality from bottom to top), cf.~\cite[Definition~5.5]{Brotherston11}.

In our example that ``derives'' $p^+ \to p$, the occurrence of $p^+$ that is active in the lower application of $({}^+\to)_\mathrm{L}$
tracks to the {\em left} premise, and the $p^+$ that goes further to the infinite path is {\em another} occurrence generated by cut.
For the cut-free system, this constraint holds automatically.

Also notice that the rules in $\Linfty$ are asymmetric: we don't introduce the rules where $A$ appears to the right of 
$A^+$. Yet, this calculus is equivalent to the symmetric system~$\Lomega$ (Proposition~\ref{Prop:inftyeq}). A motivation
for this asymmetry is explained in the end of Section~\ref{S:cyclic}.

We generalize both $\Linfty$ and $\Lomega$ by adding the additive disjunction, $\vee$, governed by the following rules:
$$
\infer[(\vee\to)]{\Gamma, A_1 \vee A_2, \Delta \to C}{\Gamma, A_1, \Delta \to C & \Gamma, A_2, \Delta \to C}
\qquad
\infer[(\to\vee)]{\Gamma \to A_1 \vee A_2}{\Gamma \to A_i}
$$
and denote the extensions by $\Linfty(\vee)$ and $\Lomega(\vee)$ respectively.

The cut-free calculi $\Lomega(\vee)$ and $\Linfty(\vee)$ (and, therefore, their conservative
fragments $\Lomega$ and $\Linfty$) are equivalent.

\begin{proposition}\label{Prop:inftyeq}
A sequent is derivable in $\Lomega(\vee)$ if{f} it is derivable in $\Linfty(\vee)$.
\end{proposition}

\begin{proof}[sketch of]
The {\em ``only if''} part is trivial: the $\omega$-rule is derivable in $\Linfty(\vee)$
and so are the $(\to{}^+)_n$ rules.
All other rules are the same.

For the {\em ``if''} part, we make use of the $*$-elimination result by Palka~\cite{Palka07}. 
We consider the {\em $n$-th negative mapping} that replaces any negative occurrence of $A^+$
(polarity is defined as usual) by $A \vee A^2 \vee \ldots \vee A^n$ and show that if a sequent is derivable
in $\Linfty(\vee)$, than all its negative mappings are also derivable. In the negative mapping, however, there are
no negative occurrences of $^+$, and therefore its cut-free derivation doesn't have infinite branches.
Moreover, we replace each $(\to{}^+)_\mathrm{L}$ rule application with the following subderivation:
$$
\infer[(\mathrm{cut})]{\Pi_1, \Pi_2 \to A^+}{\infer{\Pi_1, \Pi_2 \to A \cdot A^+}{\Pi_1 \to A & \Pi_2 \to A^+} &
\infer{A \cdot A^+ \to A^+}{A, A^+ \to A^+}}
$$
The sequent $A, A^+ \to A^+$ is derivable in $\Lomega(\vee)$, using the $\omega$-rule.
Thus, the negative mapping of the original is derivable in $\Lomega(\vee)$ using cut, and, by cut elimination,
has a cut-free derivation. Then we go backwards and show, following the argument of
Palka~\cite{Palka07}, that the original sequent is derivable in $\Lomega(\vee)$.
\end{proof}

\section{The Cyclic Calculus}\label{S:cyclic}
Now let's consider the following example:
$$
\infer[(\to\BS)]{(p \BS p)^+ \to p \BS p}
{\infer[({}^+\to)]{p, (p \BS p)^+ \to p}{p, p \BS p \to p & \infer[(\BS\to)]{p, p \BS p, (p \BS p)^+ \to p}{p \to p & \infer{p, (p \BS p)^+ \to p}{\iddots}}}}
$$
We see that actually we don't have to develop the derivation tree further, since the sequent
$p, (p \BS p)^+ \to p$ on top already appears lower in the derivation, and now
this tree can be built up to an infinite one in a regular way. 

We define the notion of
{\em cyclic proof} as done in~\cite{Shamkanov14}\cite{Shamkanov16} (for GL, the G\"odel -- L\"ob logic) and call this system~$\Lcirc$.
In contrast to the situation with GL, however, here $\Lcirc$ is {\em strictly weaker} than $\Lomega$ ($\Linfty$)
due to complexity reasons. Indeed, $\Lomega$ is $\Pi_1^0$-hard, while in $\Lcirc$ derivations are finite and
the derivability problem is recursively enumerable (belongs to $\Sigma_1^0$). This is true even in the signature
without $\vee$.

For the extension of $\Lcirc$ with additive disjunction, we show that the cyclic system $\Lcirc(\vee)$
 is equivalent to the corresponding variant of {\em action logic} considered
by Pratt~\cite{Pratt91}, Kozen~\cite{Kozen94}, and Jipsen~\cite{Jipsen04}. The difference is due to Lambek's
non-emptiness restriction and the use of positive iteration instead of Kleene star.

Formally, cyclic derivations are defined as follows. The system $\Lcirc(\vee)$ has the same axioms and rules
as $\Linfty(\vee)$, but infinite derivations are not allowed. Instead, for each application of the $({}^+\to)_{\mathrm{L}}$ rule
that yields $\Gamma, A^+, \Delta \to B$ we trace the active occurrence of $A^+$ upwards and are allowed to stop
if we again get the same sequent, $\Gamma, A^+, \Delta \to B$ with the same occurrence of $A^+$. This sequent is {\em backlinked}
to the original one, forming a cycle. The cut rule is also allowed. Note that in the bottom of each cycle we always have the 
$({}^+\to)_{\mathrm{L}}$ rule with the active occurrence of $A^+$ which is traced through the cycle, thus satisfying the constraint
needed for infinite derivations with cut. Clearly, every cyclic derivation can be expanded into an infinite one.
On the other hand, the cyclic system $\Lcirc$ is not equivalent to $\Linfty$ due to complexity reasons.

This system $\Lcirc$ appears to have much in common with various {\em coinductive} proof 
systems~\cite{Brandt98}\cite{Jaffar08}\cite{KozenSilva16}\cite{Pous11}\cite{Rosu09}.
These connections are worth further investigation.

The cyclic system $\Lcirc(\vee)$ happens to be equivalent to a non-sequential calculus $\mathbf{ACT}^+$ defined below, which is 
the positive iteration variant of the axioms for action algebras
by Pratt~\cite{Pratt91}:

$$
A \to A \qquad (A \cdot B) \cdot C \to A \cdot (B \cdot C) 
\qquad
A \cdot (B \cdot C) \to (A \cdot B) \cdot C
$$
$$
\infer{A \cdot B \to C}{A \to C \SL B}\qquad
\infer{A \to C \SL B}{A \cdot B \to C}\qquad
\infer{A \cdot B \to C}{B \to A \BS C}\qquad
\infer{B \to A \BS C}{A \cdot B \to C}
$$
$$
\infer{A \to C}{A \to B & B \to C}\qquad
\infer{A \to B_1 \vee B_2}{A \to B_i}
\qquad
\infer{A_1 \vee A_2 \to B}{A_1 \to B & A_2 \to B}
$$
$$
A \vee (A^+ \cdot A^+) \to A^+
\qquad
\infer{A^+ \to B}{A \vee (B \cdot B) \to B}
$$

The rules for $\BS$, $\SL$, and $\cdot$ correspond to the non-sequential formulation of the Lambek calculus~\cite{Lambek58}.

\begin{lemma}\label{Lm:shortrule}
The following rule is admissible in $\mathbf{ACT}^+$:
$$
\infer{A^+ \to C}{A \to C & C \cdot A \to C}
$$
\end{lemma}

This lemma is actually a modification of a well-known alternative formulation of the calculus
for action logic (connecting it to Kleene algebra). The difference, again, is in using 
positive iteration instead of Kleene star.

\begin{proof}
The second premise yields $A \to C \BS C$, and since $(C \BS C) \cdot (C \BS C) \to C \BS C$ is derivable,
we get $A \vee ((C \BS C) \cdot (C \BS C)) \to C \BS C$, and therefore $A^+ \to C \BS C$ and
then $C \to C \SL A^+$. By transitivity with $A \to C$ this yields $A \to C \SL A^+$, and therefore
$A \cdot A^+ \to C$. Combining this with $A \to C$, we get $A \vee (A \cdot A^+) \to C$, and it is sufficient
to show $A^+ \to A \vee (A \cdot A^+)$. Denote $A \vee (A \cdot A^+)$ by $B$. We have $A \to B$ and also
$B \cdot B \to B$. Indeed, using distributivity conditions: $(E \vee F) \cdot G \leftrightarrow (E \cdot G) \vee (E \cdot G)$
and $G \cdot (E \vee F) \leftrightarrow (G \cdot E) \vee (G \cdot F)$, that are derivable in $\mathbf{ACT}^+$,
 we replace $B \cdot B$ with $(A \cdot A) \vee (A \cdot A \cdot A^+) \vee
(A \cdot A^+ \cdot A) \vee (A \cdot A^+ \cdot A \cdot A^+)$, and applying the axiom for ${}^+$ and monotonicity,
we see that all four disjuncts here yield $A \cdot A^+$, and therefore $B$. Hence, by the rule
for ${}^+$, we obtain $A^+ \to B$.
\end{proof}

\begin{lemma}\label{Lm:longrule}
The following rule is admissible in $\mathbf{ACT}^+$:
$$
\infer{A^+ \to C}{A \to C & A^2 \to C & \ldots & A^k \to C & A^k \cdot C \to C}
$$
\end{lemma}

Lemma~\ref{Lm:longrule} is essential for emulating cyclic reasoning in the non-sequential calculus $\mathbf{ACT}^+$.
The $k$ parameter corresponds to the number of $({}^+\to)_\mathrm{L}$ applications in the cycle.

\begin{proof}
First we prove that $$A^+ \to A \vee A^2 \vee \ldots \vee A^k \vee (A^k)^+ \vee (A^k)^+ \cdot A \vee
\ldots \vee (A^k)^+ \cdot A^{k-1}$$ is derivable in this calculus. We denote the right-hand side of this
formula by $B$ and show $A \to B$ and $B \cdot A \to B$ (this yields $A^+ \to B$ by Lemma~\ref{Lm:shortrule}).  The first is trivial.
For the second, using distributivity conditions, we replace $B \cdot A$ with
$$A^2 \vee A^3 \vee \ldots \vee A^k \vee A^{k+1} \vee (A^k)^+ \cdot A \vee (A^k)^+ \cdot A^2 \vee \ldots
\vee (A^k)^+ \cdot A^k \vee (A^k)^+ \cdot A^{k+1}.$$
All types in this long disjunction, except $A^{k+1}$ and $(A^k)^+  \cdot A^{k+1}$, belong to the disjunction $B$
(and therefore yield $B$). For the two exceptions we have the following: $A^{k+1} \to (A^k)^+ \cdot A$ and
$(A^k)^+ \cdot A^{k+1} \to (A^k)^+ \cdot A$.

Now we prove the lemma itself by deriving $B \to C$. To do this, we need to show $H \to C$ for any disjunct $H$ in $B$. For $H = A$, \dots,
$H = A^k$ this is stated in the premises. Since that $(C \SL C) \cdot (C \SL C) \to C \SL C$ is derivable and $A^k \to C \SL C$ follows from
the last premise, we get $A^k \vee ((C \SL C) \cdot (C \SL C)) \to (C \SL C)$, and therefore $(A^k)^+ \to C \SL C$. Thus,
$(A^k)^+ \cdot C \to C$, then $C \to (A^k)^+ \BS C$, and by transitivity with $A^i \to C$ we get $(A^k)^+ \cdot A^i \to C$ for any $i = 1, \dots, k-1$.
It remains to show $(A^k)^+ \to C$. We have $(A^k)^+ \cdot A^k \to C$ and also $A^k \to C$ as a premise. One can easily prove
$(A^k)^+ \to A^k \vee ((A^k)^+ \cdot A^k)$ and thus establish $(A^k)^+ \to C$.

Finally, by transitivity from $A^+ \to B$ and $B \to C$ we obtain $A^+ \to C$.
\end{proof}

\begin{theorem}
A sequent (of the form $E \to F$) is derivable in $\Lcirc(\vee)$ if{f} it is derivable in $\mathbf{ACT}^+$.
\end{theorem}

\begin{proof}
{\em The ``if'' part} is easier. The rules operating Lambek connectives ($\cdot$, $\SL$, and $\BS$) can be emulated
in the sequential calculus due to Lambek~\cite{Lambek58}. The rules for $\vee$ in $\mathbf{ACT}^+$ directly correspond to
the rules for $\vee$ in $\Lcirc(\vee)$. 

The following cyclic derivation yields $A^+ \to B$ from $A \to B$ and $B, B \to B$, thus establishing the rule for $^+$ from
$\mathbf{ACT}^+$:
$$
\infer[({}^+\to)_{\mathrm{L}}]{A^+ \to B}
{A \to B & \infer[(\mathrm{cut})]{A, A^+ \to B}{A \to B & \infer[(\mathrm{cut})]{B, A^+ \to B}{\infer{A^+ \to B}{\vdots} & B, B \to B}}}
$$
The track of $A^+$ goes through the cycle, and the $({}^+\to)_\mathrm{L}$ rule is applied to it at every round.

Finally, for $A^+ \cdot A^+ \to A^+$, we first derive $A \cdot A^+ \to A^+$ 
(using $(\to^+)_{\mathrm{L}}$ and $(\cdot\to)$), and then, following Pratt~\cite{Pratt91}, transform it into $A^+ \cdot A^+ \to A^+$:
$$
\infer{A^+ \cdot A^+ \to A^+}
{\infer{A^+ \to A^+ \BS A^+}
{\infer{A \vee ((A^+ \BS A^+) \cdot (A^+ \BS A^+)) \to A^+ \BS A^+}
{\infer{A \to A^+ \BS A^+}{A, A^+ \to A^+} & A^+ \BS A^+, A^+ \BS A^+ \to A^+ \BS A^+}}}
$$
In this derivation we've used other rules of $\mathbf{ACT}^+$, which were previously shown to be valid in $\Lcirc(\vee)$.
Together with $A \to A^+$ (derivable using $(\to{}^+)_1$), this yields the last axiom of $\mathbf{ACT}^+$, 
$A \vee (A^+ \cdot A^+) \to A^+$.

For {\em ``only if'' part,} we first replace all cycles in the $\Lcirc(\vee)$ derivation by applications
of the rule from Lemma~\ref{Lm:longrule}. We proceed by induction on the number of cycles. For the induction step,
let the derivation end with an application of $(^{}+\to)_\mathrm{L}$, involved in a cycle. Let $k$ be the number
of applications of $({}^+\to)_{\mathrm{L}}$ to the active occurrence of $A^+$ that is tracked along this cycle.
Let the goal sequent be $\Gamma, A^+, \Delta \to B$; the same sequent appears on top of the cycle:
$$
\infer[({}^+\to)_{\mathrm{L}}]{\Gamma, A^+, \Delta \to B}
{\Gamma, A, \Delta \to B & \infer{\Gamma, A, A^+, \Delta \to B}{\infer{\vdots}{\Gamma, A^+, \Delta \to B}}}
$$
Let $C = \Gamma \BS B \SL \Delta$ (if $\Gamma$ or $\Delta$ contains more than one formula, we add $\cdot$'s between them;
if $\Gamma$ or $\Delta$ is empty, we omit the corresponding division). The sequent $\Gamma, C, \Delta \to B$ is derivable
in the Lambek calculus. Then we go down the cycle path, replacing the active $A^+$ with $A^i, C$. We start with $i = 0$
and increase $i$ each time we come across $({}^+\to)_{\mathrm{L}}$ applied to the active $A^+$. After this substitution,
this application becomes trivial: instead of
$$
\infer[({}^+\to)_{\mathrm{L}}]{\Gamma', A^+, \Delta' \to B'}{\Gamma', A, \Delta' \to B' & \Gamma', A, A^+, \Delta' \to B'}
$$
we get
$$
\infer{\Gamma', A^{i+1}, C, \Delta' \to B'}{\Gamma', A, A^i, C, \Delta' \to B'}
$$
and actually forget about the left premise of the rule. All other rules remain valid. In the end, this gives us $\Gamma, A^k, C, \Delta \to B$,
or $A^k \cdot C \to C$. Moreover, the derivation of this sequent was obtained by substitution and cutting some branches from the original
derivation, and therefore contains less cycles. By induction, we can suppose that $A^k \cdot C \to C$ was derived without cycles, using
the rule from Lemma~\ref{Lm:longrule}.

Next, for an arbitrary $j$ from 1 to $k$, we go upwards along the trace of the active $A^+$ and find the $j$-th application of $({}^+\to)_{\mathrm{L}}$:
$$
\infer{\Gamma', A^+, \Delta' \to B'}{\Gamma', A, \Delta' \to B' & \Gamma', A, A^+, \Delta' \to B'}
$$
Now we cut off the right (cyclic) derivation branch and replace $A^+$ in the goal with $A$. Next, we trace it down
back to the original sequent, replacing $A^+$ with $A^i$. The index $i$ starts from 1 and gets increased each time we
pass through the $({}^+\to)_{\mathrm{L}}$ rule with the active $A^+$. Again, these applications trivialize, all other rules
remain valid. In the end, we get $\Gamma, A^j, \Delta \to B$ derivable with a less number of cycles. This yields $A^j \to C$.

Finally, having $A \to C$, $A^2 \to C$, \dots, $A^k \to C$, and $A^k \cdot C \to C$, we apply Lemma~\ref{Lm:longrule} and obtain
$A^+ \to C$. Using cut, we invert $(\cdot\to)$, $(\to\SL)$, and $(\to\BS)$, decompose $C$ and arrive at the original goal
sequent $\Gamma, A^+, \Delta \to B$.

This finishes the non-trivial part of the proof: now we have a normal, non-cyclic derivation, and it remains to show that other rules
of $\Lcirc(\vee)$ used in it are admissible in $\mathbf{ACT}^+$. (Formally speaking, the languages of $\Lcirc(\vee)$ and $\mathbf{ACT}^+$ are
different. In $\mathbf{ACT}^+$, instead of sequents of the form $A_1, \dots, A_n \to B$, we consider $A_1 \cdot \ldots \cdot A_n \to B$.)

The rules for Lambek connectives ($\BS$, $\SL$, and $\cdot$), and also the cut rule,
are admissible in $\mathbf{ACT}^+$ due to Lambek~\cite{Lambek58}. The rules for $\vee$
correspond directly. Finally, the $(\to{}^+)_1$ and $(\to{}^+)_{\mathrm{L}}$ are validated as follows (here we use previously validated Lambek rules):
$$
\infer{\Pi \to A^+}{\Pi \to A &
\infer{A \to A^+}{\infer{A \to A \vee (A^+ \cdot A^+)}{A \to A} & A \vee (A^+ \cdot A^+) \to A^+}}
$$
$$
\infer{\Pi_1, \Pi_2 \to A^+}
{\infer{\Pi_1, \Pi_2 \to A^+ \cdot A^+}{\infer{\Pi_1 \to A^+}{\Pi_1 \to A} & \Pi_2 \to A^+} & 
\infer{A^+ \cdot A^+ \to A^+}{\infer{A^+ \cdot A^+ \to A \vee (A^+ \cdot A^+)}{A^+ \cdot A^+ \to A^+ \cdot A^+} & A \vee (A^+ \cdot A^+) \to A^+}}
$$
\end{proof}

Note that, despite the fact that the calculus for $\mathbf{ACT}^+$ is symmetric, the asymmetry in the rules of
$\Lcirc(\vee)$ is essential for our reasoning, because if we allow both left and right rules for ${}^+$, the rule from Lemma~\ref{Lm:longrule},
that is used to emulate cyclic derivation, would transform into
$$
\infer{A^+ \to C}{A \to C & A^2 \to C & \dots & A^k \to C & A^\ell \cdot C \cdot A^{k-\ell} \to C}
$$
and for this rule we don't know whether it is admissible in $\mathbf{ACT}^+$.

\section{Further Work and Open Questions}
In this section we summarize the questions that are still (to the author's best knowledge) unsolved.

\begin{enumerate}
\item Though we don't claim cut elimination for $\Linfty$ in this paper, it looks plausible that it could be proven using
continuous cut elimination (cf.~\cite{Mints78}\cite{SavateevShamkanov}). For $\Lcirc$, however, the problem looks harder, since if one
unravels the cyclic derivation into an infinite one and eliminates cut, the resulting derivation could be not cyclic anymore.
\item In this paper we use complexity arguments to show that $\Lomega$ is strictly more powerful than any its
subsystem with finite derivations. This doesn't yield any examples of concrete sequents derivable in $\Lomega$ and
not derivable, say, in $\Lcirc$. Constructing such examples is yet an open problem.
\item We don't know whether the rule in the end of Section~\ref{S:cyclic} is admissible if $\mathbf{ACT}^+$. If yes,
we could allow both left and right rules for ${}^+$ is cyclic derivations, and this system would be still equivalent to
$\mathbf{ACT}^+$.
\item Safiullin's construction (see Appendix) essentially uses Lambek's non-empti\-ness restriction. The question whether
any context-free language can be generated by a categorial grammar with unique type assignment, based on the variant
of the Lambek calculus allowing empty left-hand sides of sequents, is still open. From our perspective, a positive answer
to this question (maybe, by modification of Safiullin's construction) would immediately yield $\Pi_1^0$-hardness of
the Lambek calculus allowing empty left-hand sides of sequents, enriched with Kleene star (but without additive conjunction
and disjunction), thus solving a problems posed by Buszkowski~\cite{Buszkowski07}.
\item An open (and, in the view of the sophisticatedness of Pentus' completeness proof~\cite{Pentus95}, very hard)
question is the completeness of $\Lomega$ w.r.t. language interpretation (see Section~1). A partial completeness result,
for the fragment where $^+$ is allowed only in the denominators of $\BS$ and $\SL$, was obtained by Ryzhkova~\cite{Ryzhkova13},
using Buszkowski's canonical model construction~\cite{Buszkowski82}.
\end{enumerate}

\subsection*{Acknowledgments} 
The author is grateful to Arnon Avron, Lev Beklemishev, Michael Kaminski, Max Kanovich, Glyn Morrill, Fedor Pakhomov, 
Mati Pentus, Nadezhda Ryzhkova, Andre Scedrov, Daniyar Shamkanov, and Stanislav Speranski for fruitful discussions.
The author is also grateful to the anonymous referees for their comments.

\section*{Appendix: Safiullin's Construction Revisited}

Theorem~\ref{Th:Safiullin} by Safiullin is a crucial component of our $\Pi_1^0$-hardness proof for $\Lomega$. Unfortunately,
Safiullin's paper~\cite{Safiullin07} is very brief and, moreover, includes this theorem (which is probably the most
interesting result of that paper) as a side-effect of a more complicated construction. This makes it very hard to follow Safiullin's ideas
and arrive at a complete proof. Therefore, in this Appendix we present Safiullin's proof clearly and in detail.

In this Appendix, the ${}^+$ connective is never used, and $\Tp$ stands for the set of types constructed from
primitive ones using $\cdot$, $\BS$, and $\SL$.

Define the {\em top} of a Lambek type in the following way:
$
\Top(q) = q \text{ for $q \in \PR$}$; $
\Top(A \BS B)= \Top(B \SL A) = \Top(B).
$
Note that the $A \cdot B$ case is missing. Thus, not every type has a top.

For types with tops, the $(\to\cdot)$ rule is invertible (proof by induction): 
\begin{lemma}\label{Lm:cdot_inv}
If all types of $\Pi$ have tops and $\Pi \to A_1 \cdot \ldots \cdot A_n$ is derivable in $\LL$, then
$\Pi = \Pi_1, \dots, \Pi_n$ and $\Pi_i \to A_i$ is derivable for every $i = 1, \ldots, n$.
\end{lemma}

If a sequent of the form $\Pi \to q$, $q \in \PR$, has a cut-free derivation in $\LL$, trace the occurrence of $q$
back to the axiom of the form $q \to q$, and then trace the left $q$ back to its occurrence in $\Pi$. This occurrence
of $q$ will be called the {\em principal} occurrence (for different derivations, the principal occurrences could
differ).

\begin{lemma}\label{Lm:principal}
The principal occurrence has the following properties:
\begin{enumerate}
\item if all types in $\Pi$ have tops, then the principal occurrence is one of them;
\item if in a derivation of $\Pi, q, \Phi \to q$ the occurrence of $q$ between $\Pi$ and $\Phi$
is principal, then $\Pi$ and $\Phi$ are empty;
\item if in a derivation of $\Pi, q \SL A, \Phi \to q$ the occurrence of $q$ in $q \SL A$ is principal,
then $\Pi$ is empty;
\item if in a derivation of $\Pi, A \BS q, \Phi \to q$ the occurrence of $q$ in $A \BS q$ is principal,
then $\Phi$ is empty.
\end{enumerate}
\end{lemma}

\begin{proof}
For statement~1, proceed by induction on derivation. For statements~2--4, suppose the contrary
and also proceed by induction on derivation.
\end{proof}

\begin{lemma}\label{Lm:qq1}
If all types of $\Pi$ have tops, and these tops are not $q$, then $\Pi \to q \SL q$ is not derivable in $\LL$.
\end{lemma}
\begin{proof}
Since $(\to\SL)$ is invertible, we get $\Pi, q \to q$, and by Lemma~\ref{Lm:principal} $\Pi$ should be empty.
But $\to q \SL q$ is not derivable due to Lambek's restriction.
\end{proof}

\begin{lemma}\label{Lm:principal_decomposition}
If in a derivation of $q \SL A, \Phi \to q$ the leftmost occurrence of $q$ is principal or in a derivation of
$\Phi, A \BS q \to q$ the rightmost occurrence of $q$ is principal, then $\Phi \to A$ is derivable.
\end{lemma}

\begin{proof}
Induction on the derivation. 
\end{proof}

\begin{lemma}\label{Lm:qq2}
If $\Pi, q \SL q, \Phi \to q \SL q$ is derivable in $\LL$, all types from $\Pi$ and $\Phi$ have
tops, and these tops are not $q$, then $\Pi$ and $\Phi$ are empty.
\end{lemma}

\begin{proof}
Again, by inverting $(\to\SL)$ we get $\Pi, q \SL q, \Phi, q \to q$. The rightmost $q$ cannot be principal,
because otherwise $\Pi, q \SL q, \Phi$ is empty (Lemma~\ref{Lm:principal}). The second possibility is the 
top of $q \SL q$. Then, again by Lemma~\ref{Lm:principal}, $\Pi$ is empty, and by Lemma~\ref{Lm:principal_decomposition}
$\Phi, q \to q$ is derivable. Since tops of $\Phi$ are not $q$, the rightmost occurrence of $q$ is principal.
By Lemma~\ref{Lm:principal}, $\Phi$ is empty.
\end{proof}

By $\FG$ we denote the free group generated by the set of primitive types $\PR$.
For every $A \in \Tp$ we define its interpretation in this free group, $\FGI{A}$, as follows:
$
\FGI{q} = q \text{ for $q \in \PR$}$;
$\FGI{A \cdot B} = \FGI{A} \FGI{B}$;
$\FGI{A \BS B} = \FGI{A}^{-1} \FGI{B}$;
$\FGI{B \SL A} = \FGI{B} \FGI{A}^{-1}$.
If $\FGI{A}$ is the unit of $\FG$, $A$ is called a {\em zero-balance type.}

The {\em primitive type} count, $\#_q(A)$, for $q \in \PR$ and $A \in \Tp$, is defined as follows:
$\#_q(q) = 1$; $\#_q(q') = 0$, if $q' \in \PR$ and $q' \ne q$;
$\#_q(A \cdot B) = \#_q(A) + \#_q(B)$; $\#_q(A \BS B) = \#_q(B \SL A) = \#_q(B) - \#_q(A)$.
Notice that if $A$ is a zero-balance type, then $\#_q(A) = 0$ for every $q \in \PR$.

If the sequent $A_1, \dots, A_n \to B$ is derivable in $\LL$, then it is {\em balanced,} namely, $\#_q(B) = \#_q(A_1) + \ldots + \#_q(A_n)$ for
every $q \in \PR$, and $\FGI{A_1} \ldots \FGI{A_n} = \FGI{B}$.

\begin{theorem}[M. Pentus, 1994]\label{Th:conjoinability}
If $\FGI{A_1} = \FGI{A_2} = \ldots = \FGI{A_n}$, then there exists such $B \in \Tp$, that
all sequents $A_1 \to B$, $A_2 \to B$, \ldots, $A_n \to B$ are derivable in $\LL$.~{\rm\cite{Pentus94}}
\end{theorem}

For a set of zero-balance types $\Uc = \{ A_1, \dots, A_n \}$, we construct an ersatz of their additive disjunction,
$A_1 \vee \ldots \vee A_n$, in the following way. In the notations for types, we sometimes omit
the multiplication sign, $\cdot$, if this doesn't lead to misunderstanding. Let $u$, $t$, and $s$ be fresh primitive types,
not occurring in $A_i$. By Theorem~\ref{Th:conjoinability}, there exist such types $F$ and $G$ that the folllowing sequents
are derivable for all $i = 1, \ldots, n$:
$$
(t \SL t) A_i (t \SL t) \ldots (t \SL t) A_n (t \SL t) \to F,
\qquad
(t \SL t) A_1 (t \SL t) \ldots (t \SL t) A_i (t \SL t) \to G.
$$
Now let 
$$
E = (t \SL t) A_1 (t \SL t) A_2 (t \SL t) \ldots (t \SL t) A_n (t \SL t),
$$
$$
B = E\, (((u \SL F) \BS u) \BS (t \SL t)),
\qquad
C = ((t \SL t) \SL (u \SL (G \BS u))) \, E.
$$
We omit the multiplication sign, $\cdot$, if this doesn't lead to misunderstanding.

Finally,
$
\IS(\Uc) = ((s \SL E) \cdot B) \BS s \SL C.
$

\begin{lemma}\label{Lm:is1}
For each $A_i \in \Uc$, the sequent $A_i \to \IS(\Uc)$ is derivable in $\LL$.
\end{lemma}

\begin{proof}
The derivation is straightforward.
\end{proof}

\begin{lemma}\label{Lm:is2}
If the sequent $\Pi \to \IS(\Uc)$ is derivable in $\LL$, all types in $\Pi$ have tops, and these tops are not $s$ or $t$, then
for some $A_j \in \Uc$ the sequent $B_2, \Pi, C_1 \to A_j$, where $B_2$ is either empty
or is a type such that $B = B_2$ or $B = B_1 \cdot B_2$ for some $B_1$, and $C_1$
is either empty or is a type such that $C = C_1$ or $C = C_1 \cdot C_2$ for some $C_2$ (up to associativity of $\cdot$).
\end{lemma}
Using the invertibility of $(\cdot\to)$, we replace $\cdot$'s in $B_2$ and $C_1$ by commas, and thus consider them
as sequences of types that have tops. Actually, we want them to be empty, and it will be so in our final construction.

\begin{proof}
Let $\Pi \to \IS(\Uc)$ be derivable. Then one can derive
$(s \SL E), B, \Pi, C \to s$, and then by Lemma~\ref{Lm:principal_decomposition} we get
$B, \Pi, C \to E$ (since the leftmost $s$ is the only top $s$, and it is the principal occurrence).
Recall that $E = (t \SL t) A_1 \dots (t \SL t) A_n (t \SL t)$ and apply Lemma~\ref{Lm:cdot_inv}. It is
sufficient so show that, after decompositon, the whole $\Pi$ comes to one part of the left-hand side of the sequent. Suppose
the contrary, then locate the principal occurrence of $t$ (it should be in $B$). Then proceed by induction: 
finally we run out of $t$'s in $B$ and get a contradiction.
\end{proof}

\begin{proof}[of Theorem~\ref{Th:Safiullin}]
Given a context-free grammar $\Gc$ without $\varepsilon$-rules, we need to construct an equivalent Lambek grammar with unique type assignment. Let
$\Sigma = \{ a_1, \dots, a_\mu \}$ be the alphabet,
$\Nc = \{ N_0, N_1, N_2, \dots, N_\nu \}$ be the set of non-terminal symbols of $\Gc$, $N_0$ is the starting symbol. 

First we algorithmically transform $\Gc$ into Greibach normal form~\cite{Greibach65} with rules of the following three forms:
$N_i \Rightarrow a_j N_k N_\ell$, $N_i \Rightarrow a_j N_k$, or $N_i \Rightarrow a_j$.

Now we construct the Lambek grammar.  Let $\PR$ include
distinct primitive types $p$, $p_1, \dots, p_\nu$, $r$, $u$, $t$, and $s$.
For each $i = 0, \dots, \nu$ let $H_i = p \SL ((p_i \SL p_i) \cdot p)$ (this type corresponds to the non-terminal $N_i$).
Next, for each $j = 1, \dots, \mu$, we form a set $\Uc_j$ in the following way:
\begin{align*}
\text{add } K_{i,k,\ell} = r \SL \bigl( ( H_k \cdot H_\ell \cdot (p_i \SL p_i)) \BS r \bigr) & \text{ for each rule $N_i \Rightarrow a_j N_k N_\ell$,}\\
\text{add } K_{i,k} = r \SL \bigl( ( H_k \cdot (p_i \SL p_i)) \BS r \bigr) & \text{ for each rule $N_i \Rightarrow a_j N_k$,}\\
\text{add } K_i = r \SL \bigl( (p_i \SL p_i) \BS r \bigr) & \text{ for each rule $N_i \Rightarrow a_j$.}
\end{align*}

Now let $D_j = \IS(\Uc_j)$ and $A_j = p \SL (D_j \cdot p)$ be the type corresponding to $a_j$. For the target type $H$ we take $H_0$.
Our {\bf claim} is that $a_{i_1} \dots a_{i_n} \in \Lc(\Gc)$ if{f} the sequent $A_{i_1}, \dots, A_{i_n} \to H_0$ is derivable in $\LL$.

For the easier {\em ``only if'' part,} we prove a more general statement: if $\gamma \in (\Nc \cup \Sigma)^+$ can be generated from
$N_m$ in $\Gc$, then the sequent $\Gamma \to H_m$ is derivable in $\LL$, where $\Gamma$ is a sequence of types corresponding to letters
of $\gamma$, $A_j$ for $a_j \in \Sigma$ and $H_i$ for $N_i \in \Nc$. To establish this, it is sufficient to prove that the following sequents
are derivable (each step of the context-free generation maps to a $(\mathrm{cut})$ with the corresponding sequent):
\begin{align*}
A_j, H_k, H_\ell \to H_i & \text{ for each rule  $N_i \Rightarrow a_j N_k N_\ell$,}\\
A_j, H_k \to H_i & \text{ for each rule $N_i \Rightarrow a_j N_k$,}\\
A_j \to H_i & \text{ for each rule $N_i \Rightarrow a_j$.}
\end{align*}
Consider sequents of the first type (the second and the third types are handled similarly). Since $D_j = \IS(\Uc_j)$
and $K_{i,k,\ell} =   r \SL \bigl( ( H_k \cdot H_\ell \cdot (p_i \SL p_i)) \BS r \bigr) \in \Uc_j$, we have $K_{i,k,\ell} \to D_j$ by Lemma~\ref{Lm:is1}. 
Then the derivation is as follows:
$$
\infer{p \SL (D_j \cdot p), H_k, H_\ell \to p \SL ((p_i \SL p_i) \cdot p)}
{\infer{p \SL (D_j \cdot p), H_k, H_\ell, p_i \SL p_i, p \to p}
{\infer{H_k, H_\ell, p_i \SL p_i, p \to D_j \cdot p}
{\infer[(\mathrm{cut})]{H_k, H_\ell, p_i \SL p_i \to D_j}
{\infer{H_k, H_\ell, p_i \SL p_i \to K_{i,k,\ell}}
{H_k, H_\ell, p_i \SL p_i, (H_k \cdot H_\ell \cdot (p_i \SL p_i)) \BS r \to r} 
& K_{i,k,\ell} \to D_j}  & p \to p} & p \to p}}
$$

For the {\em ``if'' part,} let $A_{i_1}, \dots, A_{i_n} \to H_i$ be derivable and proceed by induction on the cut-free derivation
($i$ is arbitrary here for induction; in the end $i = 0$). Since $H_i = p \SL ((p_i \SL p_i) \cdot p)$ and $(\to\SL)$ and $(\cdot\to)$ are
invertible, we get $A_{i_1}, \dots, A_{i_n}, p_i \SL p_i, p \to p$. 
Locate the principal occurrence of $p$. By Lemma~\ref{Lm:principal}, it is the $p$ in
$A_{i_1} = p \SL (D_{i_1} \cdot p)$, and by Lemma~\ref{Lm:principal_decomposition} the sequent
$A_{i_2}, \dots, A_{i_n}, p_i \SL p_i, p \to D_{i_1} \cdot p$ is derivable.
Let $j = i_1$.
Since all our types have tops, apply Lemma~\ref{Lm:cdot_inv}.

{\em Case 1 (good).} The sequent decomposes into $A_{i_2}, \dots, A_{i_n}, p_i \SL p_i \to D_j$ and $p \to p$.
Consider the former sequent. Tops on the left side are $p$ and $p_i$, we can apply Lemma~\ref{Lm:is2} and get
$B_2, A_{i_2}, \dots, A_{i_n}, p_i \SL p_i, C_1 \to K$ for some $K \in \Uc_j$.

Let's prove that $B_2$ and $C_1$ in this case are empty. Suppose $K = K_{i',k,\ell}$ (the cases of $K_{i',k}$ and $K_{i'}$ are handled
similarly). Then, by invertibility of $(\to\SL)$, we get $B_2, A_{i_2}, \dots, A_{i_n}, p_i \SL p_i, C_1, (H_k \cdot H_\ell \cdot (p_{i'} \SL p_{i'})) \BS r \to r$.
Now locate the principal occurrence of $r$.

{\em Subcase 1.1.} The principal occurrence of $r$ is the rightmost one. By Lemma~\ref{Lm:principal_decomposition}, 
we get $B_2, A_{i_2}, \dots, A_{i_n}, p_i \SL p_i, C_1 \to H_k \cdot H_\ell \cdot (p_{i'} \SL p_{i'})$.
Apply Lemma~\ref{Lm:cdot_inv}. 
First, by Lemma~\ref{Lm:qq1}, $i = i'$, otherwise there's no $p_i$ in tops of the left-hand side.
Next,  the part of the left-hand side that yields $(p_i \SL p_i)$, by Lemma~\ref{Lm:qq2},
contains only $(p_i \SL p_i)$.
 Therefore, $C_1$ is empty. Now, for some part $\Pi$ we have $\Pi \to H_k$, and, decomposing $H_k$,
we get $\Pi, p_i \SL p_i, p \to p$. By Lemma~\ref{Lm:principal}, the principal occurrence of $p$ is not the rightmost one. Since $B_2$ doesn't
have $p$ in tops, $\Pi$ should include also some of the $A_{i_2}, \dots$, and the principal $p$ is the top of one of them.
But, since $A_m$ is of the form $p \SL \dots$, by Lemma~\ref{Lm:principal} the part to the left of this $A_m$, and, therefore,
$B_2$ should be empty.

{\em Subcase 1.2.} The principal occurrence of $r$ is somewhere in $B_2$ or $C_1$, in a type $K \in \Uc_j$.
By Lemma~\ref{Lm:principal}, it is then the leftmost occurrence of $r$, because $K$ has the form $r \SL \ldots$.
This rules out the possibility of it being in $C_1$ (we definitely have $p_i \SL p_i$ to the left of it).
If it is in $B_2$, again by Lemma~\ref{Lm:principal_decomposition}, we get 
$B'_2, A_{i_2}, \dots, A_{i_n}, p_i \SL p_i, C_1 \to H_{k'} \cdot H_{\ell'} \cdot (p_i \SL p_i)$ ($H_{k'}$ and $H_{\ell'}$ are optional),
and we're in the same situation, as Subcase~1.1. Thus, $B'_2$ and $C_1$ should be empty. However, $B'_2$ now should contain the last
type of $B$, $(((u \SL F)\BS u) \SL (t \SL t))$. Contradiction. Subcase 1.2 impossible.

Now we have $A_{i_2}, \dots, A_{i_n}, p_i \SL p_i \to H_k \cdot H_\ell \cdot (p_i \SL p_i)$ (the only choice for the principal $r$ now is
the rightmost one, and we've applied Lemma~\ref{Lm:principal_decomposition}). By Lemma~\ref{Lm:cdot_inv}, we get
$A_{i_2}, \dots, A_{i_z} \to H_k$, $A_{i_{z+1}}, \dots, A_{i_n} \to H_\ell$, $p_i \SL p_i \to p_i \SL p_i$ (the last left
side is $p_i \SL p_i$ alone by Lemma~\ref{Lm:qq2}).

Apply induction hypothesis. In the context-free grammar, we now have $N_k \Rightarrow^* a_{i_2} \dots a_{i_z}$ and
$N_\ell \Rightarrow^* a_{i_{z+1}} \dots a_{i_n}$. Since $i' = i$, we also have the rule $N_i \Rightarrow a_j N_k N_\ell$ ($N_k$ and
$N_\ell$ are optional) in the grammar (since the corresponding $K$ type was in $\Uc_j$). Thus,
$N_i \Rightarrow a_j a_{i_2} \dots a_{i_z} a_{i_{z+1}} \dots a_{i_n}$. Recall that $j = i_1$.

{\em Case 2 (bad).} The sequent decomposes in another way, yielding $A_{i_2}, \dots, A_{i_z} \to D_j$ and
$\dots, p_i \SL p_i, p \to p$. Again, by Lemma~\ref{Lm:is2}, we get $B_2, A_{i_2}, \dots, A_{i_z}, C_1 \to K$ for some $K \in \Uc_j$,
and further $B_2, A_{i_2}, \dots, A_{i_z}, C_1, (H_k \cdot H_\ell \cdot (p_{i'} \SL p_{i'})) \BS r \to r$. Now we locate the principal
occurrence of $r$ and proceed as in Case~1. The only difference, however, is that now there is no $p_i \SL p_i$ in the left-hand
side, and for that reason derivation fails by Lemma~\ref{Lm:qq1}. Thus, Case~2 is impossible.

\end{proof}

\end{document}